\documentclass[english]{article}
\usepackage[T1]{fontenc}
\usepackage[latin9]{inputenc}
\usepackage{verbatim}
\usepackage{amsthm}
\usepackage{amsmath}
\usepackage{amssymb}
\usepackage{graphicx}
\usepackage{esint}

\makeatletter
  \theoremstyle{plain}
  \newtheorem*{thm*}{\protect\theoremname}
\theoremstyle{plain}
\newtheorem{thm}{\protect\theoremname}[section]
  \theoremstyle{remark}
  \newtheorem{rem}[thm]{\protect\remarkname}
  \theoremstyle{definition}
  \newtheorem{defn}[thm]{\protect\definitionname}
  \theoremstyle{plain}
  \newtheorem{lem}[thm]{\protect\lemmaname}
  \theoremstyle{plain}
  \newtheorem{prop}[thm]{\protect\propositionname}
  \theoremstyle{plain}
  \newtheorem{cor}[thm]{\protect\corollaryname}
  \theoremstyle{definition}
  \newtheorem{example}[thm]{\protect\examplename}
  \theoremstyle{remark}
  \newtheorem{claim}[thm]{\protect\claimname}
  \theoremstyle{plain}
  \newtheorem{fact}[thm]{\protect\factname}
\newcommand{\lyxaddress}[1]{
\par {\raggedright #1
\vspace{1.4em}
\noindent\par}
}

\usepackage[all]{xy}

\makeatother

\usepackage{babel}
  \providecommand{\claimname}{Claim}
  \providecommand{\corollaryname}{Corollary}
  \providecommand{\definitionname}{Definition}
  \providecommand{\examplename}{Example}
  \providecommand{\factname}{Fact}
  \providecommand{\lemmaname}{Lemma}
  \providecommand{\propositionname}{Proposition}
  \providecommand{\remarkname}{Remark}
  \providecommand{\theoremname}{Theorem}
\providecommand{\theoremname}{Theorem}

\begin{document}
\global\long\def\Rad{\mathcal{R}}

\global\long\def\R{\mathbb{R}}

\global\long\def\C{\mathbb{C}}

\global\long\def\sign{\mbox{sign}}

\title{The quotient girth of normed spaces, and an extension of Sch\"{a}ffer's
dual girth conjecture to Grassmannians}

\author{Dmitry Faifman%
\thanks{Partially supported by ISF grant 387/09%
}}
\maketitle
\begin{abstract}
In this note we introduce a natural Finsler structure on convex surfaces,
referred to as the quotient Finsler structure, which is dual in a
sense to the inclusion of a convex surface in a normed space as a
submanifold. It has an associated quotient girth, which is similar
to the notion of girth defined by Sch\"{a}ffer. We prove the analogs
of Sch\"{a}ffer's dual girth conjecture (proved by \'{A}lvarez-Paiva)
and the Holmes-Thompson dual volumes theorem in the quotient setting.
We then show that the quotient Finsler structure admits a natural
extension to higher Grassmannians, and prove the corresponding theorems
in the general case. We follow \'{A}lvarez-Paiva's approach to the
problem, namely, we study the symplectic geometry of the associated
co-ball bundles. For the higher Grassmannians, the theory of Hamiltonian
actions is applied.
\end{abstract}
In the following, we will be concerned with certain invariants of
Finsler manifolds that are associated naturally to real, finite-dimensional
normed spaces. For a survey of Sch\"{a}ffer's work on the subject,
and some related facts from convex geometry, see \cite{Th}.\\

Consider a normed space $V$. Let $M\subset V$ be a closed hypersurface.
As a submanifold of $V$, $M$ inherits a natural Finsler structure
(that is, a norm on every tangent space), denoted $\psi^{V}$. We
call $(M,\psi^{V})$ the \emph{immersion} Finsler structure on $M$,
and write $\psi_{m}^{V}(v)=\|v\|_{V}$ for $v\in T_{m}M$.\\

However, if $M$ is the boundary of a strictly star-shaped body with
a center at the origin, then there is another, equally natural Finsler
structure $\phi^{V}$ on $M$, induced by $V$ (here by strictly star-shaped
we mean that $m\notin T_{m}M$ for all $m\in M$). Namely, we identify
the tangent space $T_{m}M$ with the quotient space $V/\langle m\rangle$,
where $\langle m\rangle$ is the 1-dimensional linear space spanned
by $m$. We call $(M,\phi^{V})$ the \emph{quotient }Finsler structure
on $M$, and may write for $v\in T_{m}M$ 
\[
\phi_{m}^{V}(v)=\inf_{t\in\R}\|v+tm\|_{V}
\]
If $M=S(V)$ is the unit sphere of $V$, what we get is a Finsler
structure on $M=S^{n-1}$ (as smooth manifolds), intrinsically attached
to the normed space $V$. It is the quotient structure and its generlizations
that we will study throughout the paper. In Figure 1 below, the two
Finsler structures on $S(V)$ are depicted simultaneously, and $B_{q}$
denotes the unit ball of the corresponding norm.\\
\begin{figure}[!tph]
\textbf{\emph{Immersion and quotient unit balls}}

\centering{}\caption{\protect\includegraphics{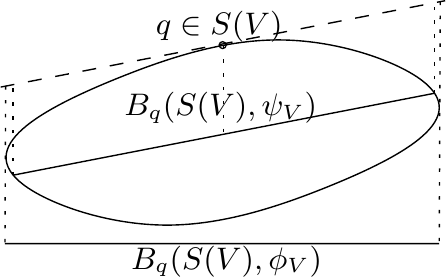}}
\end{figure}
\\
The girth of a Finsler structure on a sphere is the length of the
shortest symmetric curve. We may already state our first theorem:
\begin{thm*}
Let $V$ be a normed space. Then the girth of $(S(V),\phi^{V})$ equals
the girth of $(S(V^{*}),\phi^{V^{*}})$.
\end{thm*}
This should be compared with Sch\"affer's dual girth conjecture \cite{Sc},
proved by \'{A}lvarez-Paiva in \cite{AP}: 
\begin{thm*}
(\textbf{\emph{Sch\"affer, \'{A}lvarez-Paiva}}). Let $V$ be a normed
space. Then the girth of $(S(V),\psi^{V})$ equals the girth of $(S(V^{*}),\psi^{V^{*}})$.
\end{thm*}
Next we consider a more general scenario. Let $V$ be a linear space,
and $K,L\subset V$ convex bodies (by a convex body we always mean
the unit ball of some symmetric norm). Let $V_{L}$ denote the normed
space $V$ with unit ball $L$. Then $\partial K$ has two Finsler
structures induced on it from $V_{L}$: the immersion structure $\psi_{L}$,
and the quotient structure $\phi_{L}$. We prove a theorem concerning
the quotient structure, that parallels the generalized dual girth
conjecture proved in \cite{AP}, and the Holmes-Thompson dual volumes
theorem \cite{HT}, which in turn concern the immersion Finsler structure.
More precisely, we prove
\begin{thm*}
Let $K,L\subset V$ be convex bodies. Then \textup{$(\partial K,\phi_{L})$}
and $(\partial L^{o},\phi_{K^{o}})$ have equal girth and Holmes-Thompson
volume. If $K,L$ are smooth and strictly convex, then the spectra
and the symmetric length spectra coincide.
\end{thm*}
\noindent For the definitions of Holmes-Thompson volume, length spectrum
etc., see section \ref{sec:Background-from-Finsler}.\\

As it turns out, the quotient Finsler structure admits a natural extension
to higher Grassmannians, together with the duality theorems stated
above (for the immersion structure, no such extension is known). Consider
the oriented Grassmannian $\tilde{G}(V,k)$, which is the set of oriented
$k$-dimensional subspaces of the $n$-dimensional space $V$, with
the natural smooth manifold structure. Assume some norm $\beta$ is
given on the space of linear operators $Hom(V,V)$. We describe an
associated Finsler structure on $\tilde{G}(V,k)$, denoted $\phi_{\beta}$,
which is defined for every $1\leq k\leq n-1$; for $k=1$ and $\beta$
the nuclear norm on $Hom(V_{K},V_{L})$, it reduces to the quotient
Finsler structure $(\partial K,\phi_{L})$. We then study some of
its invariants, such as the girth and Holmes-Thompson volume. Also,
we associate a certain natural number to each geodesic, called its
rank. We prove the following 
\begin{thm*}
For any norm $\beta$ on $Hom(V,V)$, the Finsler manifolds $\tilde{(G}(V,k),\phi_{\beta})$
and $\tilde{(G}(V,n-k),\phi_{\beta})$ have equal Holmes-Thompson
volumes and equal girth. If $\beta$ is smooth and strictly convex,
the length spectra and symmetric length spectra of the two manifolds
coincide, alongside with the corresponding ranks of geodesics.
\end{thm*}
In contrast to the case $k=1$, where the result follows from the
existence of a natural diffeomorphism between the co-sphere bundles
which preserves the canonic 1-form, for general $k$ we are only able
to construct such a morphism on a dense open subset. We then make
use of a natural decomposition of the cotangent bundles, and apply
an implicit existence lemma from linear algebra.

\subsection*{Acknowledgments}

I would like to thank Vitali Milman for introducing me to the topic,
and for the stimulating talks and constant encouragement. I am indebted
to Semyon Alesker, Egor Shelukhin and Yaron Ostrover for numerous
inspiring discussions and fruitful suggestions; I am grateful to Bo'az
Klartag, Liran Rotem and Alexander Segal for multiple useful conversations
and for helping me to verify the reasoning; and to Juan Carlos \'Alvarez-Paiva
for reading the text and helping to improve the exposition. Finally,
I would like to acknowledge the Fields Institute in Toronto, where
this work started, for the kind hospitality and wonderful working
atmosphere.

\section{Background from Finsler and symplectic geometry\label{sec:Background-from-Finsler}}

For completeness, we recall the basic definitions of Finsler geometry,
and its relation to symplectic geometry. We also fix notation that
will later be used. For more details, consult \cite{AT}.

\subsection{General properties of girth}

Recall that a smooth Finsler manifold $(M,\phi)$ is a smooth manifold
$M$, together with a function $\phi:TM\to\R$ that is smooth outside
the zero section, and restricts to a strictly convex norm on $T_{m}M$
for every $m\in M$. The notions of length of a curve and distance
on $M$ are identical to their Riemannian manifolds counterpart. 
\begin{rem}
By strictly convex, we mean that the gaussian curvature is strictly
positive. In the following, we will also make use of the term \emph{weakly
strictly convex}, which means that the boundary of the body does not
contain straight segments.
\end{rem}

\begin{rem}
\label{rem:continuous_finsler}We sometimes relax the requirement
for $\phi$ to be a smooth function with strictly convex restrictions
to tangent spaces, settling only for continuity and convex restrictions.
This allows us to treat stable invariants of Finsler manifolds, such
as girth and Holmes-Thompson volume (see definitions below), under
more general conditions.\end{rem}
\begin{defn}
Let $(M,R)$ be a smooth compact manifold with an involutive diffeomorphism
$R:M\rightarrow M$ that has no fixed points. We refer to $R$ as
the antipodal map, and also denote the natural extension $R:T^{*}M\to T^{*}M$
by the same symbol. We say that a Finsler metric $\phi$ on $M$ is
symmetric if $R$ is an isometry of $(M,\phi)$. The girth $g_{\phi}(M)$
is the length of the shortest symmetric (i.e. $R$-invariant) geodesic. \end{defn}
\begin{rem}
\label{remk:girth_is_simple}Note that $g_{\phi}(M)=2\min\{dist(p,Rp):p\in M\}$
for the following reason: take the closest pair of antipodal points
$(p,Rp)$, with $d=dist(p,Rp).$ Take a curve $\gamma$ between them
with $L(\gamma)=d$. Then $R\gamma$ also joins $(p,Rp)$ and $L(R\gamma)=d$,
and for any $q\in\gamma$, we get two curves between $q$ and $Rq$,
which have length $d$. Thus $dist(q,Rq)=d$ by minimality assumption,
and so those curves are geodesic. In particular, $\gamma\cup R\gamma$
is geodesic around $p$, i.e. a symmetric closed geodesic. Therefore,
$g_{\phi}(M)\leq2\min\{dist(p,Rp):p\in M$\}. The reverse inequality
is obvious.\end{rem}
\begin{lem}
\label{continuity}Let $\phi_{1}$, $\phi_{2}$ be symmetric Finsler
structures on $(M,R)$, such that $(1+\epsilon)^{-1}\phi_{1}\leq\phi_{2}\leq(1+\epsilon)\phi_{1}.$
Then $(1+\epsilon)^{-1}g_{\phi_{1}}(M)\leq g_{\phi_{2}}(M)\leq(1+\epsilon)g_{\phi_{1}}(M)$.\end{lem}
\begin{proof}
Denote $g_{j}=g_{\phi_{j}}(M)$. Consider the $\phi_{1}$-shortest
curve $\gamma\subset\tilde{G}(V,k)$ with $\gamma(1)=R\gamma(0)$.
Then $|L(\gamma;\phi_{2})-L(\gamma;\phi_{1})|\leq\int_{0}^{1}|\phi_{2}(\dot{\gamma})-\phi_{1}(\dot{\gamma})|dt\leq\epsilon L(\gamma;\phi_{1})$.
Therefore, $g_{2}\leq2L(\gamma,\phi_{2})\leq2(1+\epsilon)L(\gamma;\phi_{1})=g_{1}(1+\epsilon)$.
Similarly, $g_{1}\leq g_{2}(1+\epsilon)$.\end{proof}
\begin{rem}
\label{Remark Continuity of girth} The Lemma often allows one to
omit assumptions of smoothness and strict convexity of the Finsler
metric when studying the girth of spaces.
\end{rem}

\subsection{Symplectic invariants}

Let $(M,\phi)$ be a smooth Finsler manifold. A curve $\gamma$ on
$M$ is a geodesic if it locally minimizes distance. The set of lengths
of closed geodesics on $M$ is called the length spectrum, and if
$M$ is equipped with an antipodal map, the symmetric length spectrum
is the set of lengths of closed symmetric geodesics. The girth is
then the minimal element of the symmetric length spectrum. \\

Every cotangent space $T_{m}^{*}M$ is equipped with the dual norm
$\phi_{m}^{*}$. Define the co-ball bundle by $B^{*}M=\{(m,\xi)\in T^{*}M:\phi_{m}^{*}(\xi)\leq1\}$,
and similarly $S^{*}M$ is the co-sphere bundle. Denote also by $T_{0}^{*}M$
the cotangent bundle with the zero-section excluded. One has the Legendre
duality map $\mathcal{L}_{m}:T_{m}M\to T_{m}^{*}M$ (see subsection
\ref{sub:Definitions-and-basic} for the definition of Legendre transform).
For a curve $\gamma\in M$ we define its lift to $T^{*}M$, denoted
$\mathcal{L}\gamma$, by $\mathcal{L}\gamma(t)=(\gamma(t),\mathcal{L}_{\gamma(t)}(\dot{\gamma}(t)))$.
We denote the canonic 1-form and symplectic form on $T^{*}M$ by $\alpha$
and $\omega=d\alpha$. Also, we define the associated Hamiltonian
function $H_{M}(m,\xi)=\frac{1}{2}\phi^{*}(\xi)^{2}$ on $T^{*}M$.\\

There are several natural volume densities on a Finsler manifold.
We will use exclusively the Holmes-Thompson volume $v_{HT}$, defined
as the push-forward under the natural projection $B^{*}M\to M$ of
the Liouville volume form $\frac{1}{n!}\omega^{n}$ on $B^{*}M$.
Note that one only needs the Finsler structure to be continuous in
order to define the Holmes-Thompson density, and that $v_{HT}(M)=\frac{1}{n!}\int_{B^{*}M}\omega^{n}$
is continuous as a function of $\phi\in C(TM,\R)$. 
\begin{prop}
Let $M$ be a geodesically complete Finsler manifold. For a curve
$\gamma(t)\in M$, denote $\Gamma(t)=\mathcal{L}\gamma(t)\in T^{*}M$.
The following are equivalent: \begin{enumerate}

\item $\gamma(t)$ is a geodesic with arc-length parametrization.

\item $\Gamma(t)$ is a flow curve for the associated Hamiltonian
$H_{M}$ of constant energy $H_{M}=\frac{1}{2}$.

\item $\Gamma(t)$ is a flow curve for the Reeb vector field on $S^{*}M$.

\end{enumerate}

\noindent Moreover, a curve $\Gamma(t)$ satisfying either 2. or
3. is necessarily the lift of a geodesic $\gamma(t)$. 
\end{prop}
For completeness, we sketch the proof of this well-known fact.
\begin{proof}
A geodesic between $a=\gamma(t_{0})$ and $b=\gamma(t_{1})$ with
a parametrization proportional to arc-length if and only if it is
a minimizer of the energy functional $E(\gamma(t))=\frac{1}{2}\int_{t_{0}}^{t_{1}}\phi(\dot{\gamma}(t))^{2}dt$
(see \cite{Mi}, ch. 12). By the Lagrangian-Hamiltonian duality, those
lift precisely to the flow curves of $H_{M}$ on $T^{*}M$. The parametrization
is arc-length $\iff$ $\phi^{*}(\mathcal{L}\dot{\gamma})\equiv1$$\iff$
$H(\mathcal{L}\gamma)=\frac{1}{2}$. The equivalence of $2.$ and
$3.$ follows from the 2-homogeneity of $H_{M}$.
\end{proof}
Slightly abusing the standard terminology, we will call a curve $\Gamma(t)$
satisfying either condition $2.$ or $3.$ a \emph{characteristic
curve.}
\begin{cor}
\label{symplectomorphism_implies_spectrum}Suppose $M$ and $N$ are
two Finsler manifolds, and either \begin{itemize}

\item $\Phi:T_{0}^{*}M\to T_{0}^{*}N$ is a symplectomorphism such
that $\Phi^{*}H_{N}=H_{M}$; or

\item $\Phi:S^{*}M\to S^{*}N$ is a diffeomorphism s.t. $\mbox{\ensuremath{\Phi}}^{*}\alpha_{N}=\alpha_{M}$.

\end{itemize}

\noindent Then $M$ and $N$ have equal length spectrum. If $M$
and $N$ both possess antipodal maps $R_{M}$, $R_{N}$ and $\Phi R_{M}=R_{N}\Phi$
then the symmetric length spectra of $M$ and $N$ coincide as well.\end{cor}
\begin{rem}
\label{omega_implies_alpha}It is easy to verify that, given two 2-homogeneous
Hamiltonian functions $H_{M}$, $H_{N}$ on $T^{*}M$ and $T^{*}N$
respectively, and a symplectomorphism between open conic subsets $\Phi:C_{M}\to C_{N}$,
where $C_{M}\subset T^{*}M$ and $C_{N}\subset T^{*}N$, such that
$\Phi^{*}H_{N}=H_{M}$, then it must necessarily preserve the canonic
one-form: $\mbox{\ensuremath{\Phi}}^{*}\alpha_{N}=\alpha_{M}$. 
\end{rem}

\section{Quotient girth\label{sec:Projective-girth}}

\subsection{Definitions and basic properties\label{sub:Definitions-and-basic}}

Let us first introduce some notation. For a normed space $V$, $S(V)$
denotes the unit sphere, and $B=B(V)$ the unit ball of $V$. For
$q\in S(V)$, the Legendre transform $\mathcal{L}_{V}(q)\in S(V^{*})$
denotes the unique covector $\xi$ for which $\{x\in V:\xi(x)=1\}$
is the tangent hyperplane to $S(V)$ at $q$. When no confusion can
arise, we write $\mathcal{L}$ instead of $\mathcal{L}_{V}$. By a
convex body, we will always mean the unit ball of some symmetric norm
on $V$. For a convex body $K\subset V$, $V_{K}$ will denote the
normed space $V$ with unit ball $K$. Denote by $\mathcal{K}(n)$
the set of convex bodies in $\R^{n}$ equipped with the Hausdorff
metric.
\begin{defn}
For a pair of convex bodies $K,L\subset V$, denote by $\psi_{L}$
the immersion Finsler structure, induced on $\partial K$ by the embedding
$\partial K\subset V_{L}$, and by $\phi_{L}$ the quotient Finsler
structure on $\partial K$, given by the obvious identification $T_{q}(\partial K)=V_{L}/\langle q\rangle$.
The immersion girth $g_{i}(K;L)$ and the quotient girth $g_{q}(K;L)$
of $K$ with respect to $L$ is the length of the shortest closed
symmetric curve on $\partial K$ with the corresponding metric. \end{defn}
\begin{rem}
For a normed space $V$ with unit ball $B$, denote $\phi_{V}=\phi_{B,B}$
and $g_{q}(V)=g_{q}(B,B)$. The Finsler manifold $(S(V),\phi_{V})$
is an intrinsic invariant of normed spaces. Moreover, for a subspace
$U\subset V$, the intrinsic Finsler structure $\phi_{U}$ on $S(U)$
coincides with the one inherited by the inclusion $S(U)\subset S(V)$. 
\end{rem}

\begin{rem}
The notions of quotient girth and quotient Holmes-Thompson volume
extend easily to any pair of convex bodies $K,L$, without any smoothness
or strict convexity assumptions (see also Remark \ref{Remark Continuity of girth})
Thus all results stated below for the girth and Holmes-Thompson volume
extend to the general case by continuity.
\end{rem}
Let us begin by comparing the immersion and quotient Finsler metrics.
Recall that for a 2-dimensional convex body $K\subset V$, its isoperimetrix
$I_{K}\subset V$ is (up to homothety) the dual body $K^{o}$, under
the identification of $V$ and $V^{*}$ by the volume form on $V$.
\begin{prop}
Let $K,L$ be smooth and weakly strictly convex. For a pair of convex
bodies $K,L\subset V$, one has the inequality $\phi_{L}\leq\psi_{L}$
on $\partial K$. We can describe the case of equality:

\begin{enumerate} 

\item In dimension $n=2$, $(\partial K,\phi_{L})$ and $(\partial K,\psi_{L})$
coincide if and only if $L$ is homothetic to $I_{K}$. In particular,
$\phi_{V}=\psi_{V}$ on $S(V)$ if and only if $S(V)$ is a Radon
curve.

\item In dimension $n\geq3$, $(\partial K,\phi_{L})$ and $(\partial K,\psi_{L})$
coincide if and only if $K$ and $L$ are homothetic ellipsoids.

\end{enumerate}\end{prop}
\begin{proof}
The inequality is obvious. Let us show that $\phi_{L}=\psi_{L}$ if
and only if for all pairs $x\in\partial K$, $y\in\partial L$, $y\in T_{x}(\partial K)\Leftrightarrow x\in T_{y}(\partial L)$. 

Fix $x\in\partial K$, $y\in\partial L$ s.t. $y\in T_{x}(\partial K)$.
Then $\psi_{L}(y)=\|y\|_{L}$, $\phi_{L}(x)=\inf_{t\in\mathbb{R}}\|tx+y\|_{L}$.
Then $\psi_{L}(x)=\phi_{L}(x)$ if and only if $\|y+tx\|_{L}\geq\|y\|_{L}$
for all $t$, i.e. $x\in T_{y}(\partial L)$. Since this is true for
all pairs $x,w$, by the weak strict convexity of $\partial L$ the
reverse implication also follows. Note that the condition on $(K,L)$
is symmetric.

When $n=2$, it is easy to see that the equality condition guarantees
uniqueness of a body $L$ corresponding to $K$: One can write a differential
equation on the polar representation of $L$. Taking $L=I_{K}$ shows
existence, proving the case $n=2$. 

Now assume $n\geq3$. If both $K,L$ are ellipsoids, the Lemma above
applies. In the other direction, it follows from the Lemma that for
any $q\in\partial K$, the shadow boundary of $L$ in the direction
$q$ lies in the hyperplane $T_{q}\partial K\subset V$. By Blaschke's
Theorem, $L$ is an ellipsoid. By symmetry, so is $K$. Thus $K$
and $L$ define two Euclidean structures that induce the same orthogonality
relation. Therefore, they are homothetic.\end{proof}
\begin{prop}
For a normed space with $\dim V=2$ and $B=B(V)$, the quotient girth
satisfies $g_{q}(V)\geq\frac{2}{\pi}\frac{M(B)}{vr(B^{o})^{2}}$,
where $M(B)=|B\times B^{o}|$ is the Mahler volume, and $vr$ denotes
the volume ratio.\end{prop}
\begin{proof}
By continuity of all magnitudes in the inequality in $\mathcal{K}(n)$
, we may assume $B(V)$ is smooth and strictly convex. Choose some
ellipsoid $E\supset B$, which defines a Euclidean structure. Fix
some orthonormal coordinates in $V$. Denote $\partial B=\gamma(\alpha)=(x(\alpha),y(\alpha))$,
where $\alpha$ is the angle measured counterclockwise from some reference
direction. Let $\beta=\beta(\alpha)$ be the angle of the point on
$\gamma$ such that $\dot{\gamma}(\beta)||(-\gamma(\alpha))$ (positively
parallel). Solving $\dot{\gamma}(\alpha)=s(\alpha)\gamma(\beta(\alpha))+t(\alpha)\gamma(\alpha)$,
we get 
\[
\phi_{V}(\dot{\gamma}(\alpha))=s(\alpha)=\frac{\dot{y}(\alpha)x(\alpha)-\dot{x}(\alpha)y(\alpha)}{y(\beta)x(\alpha)-x(\beta)y(\alpha)}
\]
Both numerator and denominator are positive. Thus 
\[
g_{q}(V)=\int_{0}^{2\pi}s(\alpha)d\alpha=\int_{0}^{2\pi}\frac{\det(\gamma(\alpha),\dot{\gamma}(\alpha))}{\det(\gamma(\alpha),\gamma(\beta))}d\alpha
\]
The denominator is bounded from above by $|\det(\gamma(\alpha),\gamma(\beta))|\leq|\gamma(\alpha)||\gamma(\beta)|\leq1$.
Therefore, 
\[
g_{q}(V)\geq\int_{0}^{2\pi}\det(\gamma(\alpha),\dot{\gamma}(\alpha))d\alpha=2Area_{E}(B)
\]
where $Area_{E}$ denotes the Lebesgue measure, normalized so that
$Area_{E}(E)=\pi$. It remains to choose the optimal $E$: 
\[
g_{q}(V)\geq2\pi\max_{E\supset B}\frac{|B|}{|E|}=2\pi\max_{E^{o}\subset B^{o}}\frac{|B\times B^{o}|}{|E\times E^{o}|}\frac{|E^{o}|}{|B^{o}|}=\frac{2}{\pi}\frac{M(B)}{vr(B^{o})^{2}}
\]
\end{proof}
\begin{cor}
For all 2-dimensional normed spaces $V$, $4<g_{q}(V)<8$. \end{cor}
\begin{proof}
One inequality is obvious: $g_{q}(V)\leq g_{i}(V)\leq8$ and $g_{i}(V)=8$
only for the parallelogram \cite{Sc}, which has quotient girth equal
to $8\log2$ (see Appendix \ref{sub:The-projective-girth of the square}),
so in fact $g_{q}(V)<8$. For the other inequality, note that by Mahler's
conjecture for the plane \cite{Ma},\cite{Re}, $M(B)$ is uniquely
minimized by the square, while by Ball's theorem \cite{Ba} (and since
the dual of a square is a square), $vr(B^{o})$ is maximized for the
square, and the inequality above is strict for the square. Thus 
\[
g_{q}(V)>\frac{2}{\pi}\frac{8}{4/\pi}=4
\]
 \end{proof}
\begin{rem}
It seems plausible that in fact $8\log2\leq g_{q}(V)\leq2\pi$ when
$\dim V=2$, the extremal cases being the square (see \ref{sub:The-projective-girth of the square}
for the computation of its quotient girth) and the circle.
\end{rem}
From Lemma \ref{continuity} we get
\begin{cor}
\label{cor:girth continuity} $g_{q}(K;L)$ is continuous on $\mathcal{K}(n)\times\mathcal{K}(n)$.
\end{cor}

\subsection{Main theorems}
\begin{thm}
\label{thm:The-co-sphere-bundles}Let $K,L\subset V$ be smooth and
strictly convex bodies. Then there is a diffeomorphism $\Phi:S^{*}(\partial K,\phi_{L})\to S^{*}(\partial L^{o},\phi_{K^{o}})$
respecting the canonic 1-form up to sign. Also, $\Phi$ respects the
antipodal map.\end{thm}
\begin{proof}
The construction of the diffeomorphism between the corresponding co-sphere
bundles is reminiscent of the one in \cite{AP}. Observe that for
$q\in\partial K$ we have the isometric embedding $T_{q}^{*}(\partial K)=(V_{L}/q)^{*}\hookrightarrow V_{L^{o}}^{*}$.
In particular, $S_{q}^{*}(\partial K)\hookrightarrow\partial L^{o}$
and in fact $S_{q}^{*}(\partial K)=\{p\in\partial L^{o}:\langle p,q\rangle=0\}$.
Define $Z\subset V\times V^{*}$ by $Z=\{(q,p):p(q)=0\}$. We then
can identify $S^{*}(\partial K)\mbox{\ensuremath{\simeq}}(\partial K\times\partial L^{o})\cap Z$,
and by symmetry also $\mbox{\ensuremath{S^{*}(\partial L^{o})\mbox{\ensuremath{\simeq}}(\partial K\times\partial L^{o})\cap Z}}$.
Now observe that the forms $\alpha_{1}=pdq$ and $\alpha_{2}=qdp$
defined on $V\times V^{*}$ satisfy $\alpha_{1}\Big|_{Z}+\alpha_{2}\Big|_{Z}=0$,
and restrict to the canonic 1-forms on $S^{*}(\partial K)$ and $S^{*}(\partial L^{o})$,
respectively. Thus $\Phi(q,p)=(p,q)$ is the required map.\end{proof}
\begin{rem}
We see from the proof above that $\psi$ is in some sense dual to
$\phi$: one has a natural isometric isomorphism of the normed bundles
$T(\partial K,\psi_{L})\rightarrow T^{*}(\partial K^{o},\phi_{L^{o}})$,
given by the Legendre transform and the fiberwise isometric identification
$T_{q}(\partial K,\psi_{L})=T_{\mathcal{L}(q)}^{*}(\partial K^{o},\phi_{L^{o}})$.
\end{rem}
As a corollary we get
\begin{thm}
\label{thm:(Dual-spheres-have}(Dual spheres have equal quotient girth)
Let $K,L\subset V$ be convex bodies. Then \textup{$(\partial K,\phi_{L})$}
and $(\partial L^{o},\phi_{K^{o}})$ have equal girth and Holmes-Thompson
volume. If $K,L$ are smooth and strictly convex, then their length
spectra and symmetric length spectra coincide. In particular, for
any normed space $V$ we get $g_{q}(V)=g_{q}(V^{*})$.
\end{thm}

\subsection{The associated double fibration}

The following is a geometric observation relating the immersion and
quotient settings, which is not used elsewhere in the paper. 

In the proof of the original girth conjecture in \cite{AP}, the following
double fibration appears naturally \[\xymatrix{& T\ar[dl] \ar[dr]&\\ \partial K & &\partial L^o }\]
where $T\subset\partial K\times\partial L^{o}$ consists of all pairs
$(q,p)$ such that the pairing $T_{q}\partial K\times T_{p}\partial L^{o}\to\mathbb{R}$
is degenerate. 

In the quotient girth setting, a different fibration appears:

\[\xymatrix{& P\ar[dl] \ar[dr]&\\ \partial L & &\partial K^o }\]
where $P=\{(q,p)\in\partial L\times\partial K^{o}:\langle q,p\rangle=0\}$
(note that we exchanged the roles of $K,L$ here). There is in fact
a natural diffeomorphism of the double fibrations: Let $B:\partial K\times\partial L^{o}\to\partial L\times\partial K^{o}$
be given by $B(q,p)=(\mathcal{L}p,\mathcal{L}q)$. One has then: $T_{q}\partial K\times T_{p}\partial L^{o}\to\mathbb{R}$
degenerate $\iff$ $\mathcal{L}q\in T_{p}\partial L^{o}$ $\iff$
$\langle\mathcal{L}q,\mathcal{L}p\rangle=0$. Thus $B:T\to P$ is
a diffeomorphism, and it respects the double fibration structure.

\section{The oriented Grassmannian}

\subsection{Background}

We begin by recalling some basic constructions, and fixing notation.

\subsubsection{Oriented Grassmannians}

The oriented Grassmannian $\tilde{G}(V,k)$, which is the set of oriented
$k$-dimensional subspaces of $V$, is naturally a smooth manifold.
We also write $\tilde{\mathbb{P}}(V)=\tilde{G}(V,1)$, which is the
projective space of oriented lines in $V$. In the following, we always
assume that $V$ is an oriented vector space.\\

Recall the Plucker embedding $i:\tilde{G}(V,k)\rightarrow\tilde{\mathbb{P}}(\wedge^{k}V)$
given by $i(\Lambda)=p(\wedge^{k}\Lambda)$, where $p:(\wedge^{k}V)\backslash0\rightarrow\tilde{\mathbb{P}}(\wedge^{k}V)$
is the canonic projection. Take $\Lambda\in\tilde{G}(V,k)$, fix a
basis $e_{1},...,e_{k}$ of $\Lambda$, and identify $T_{\Lambda}\tilde{G}(V,k)=T_{i\Lambda}i\big(\tilde{G}(V,k)\big)\simeq Hom(\Lambda,V/\Lambda)$
by assigning $\dot{\gamma}_{f}(0)\in T_{i\Lambda}i\big(\tilde{G}(V,k)\big)$
to $f\in Hom(\Lambda,V/\Lambda)$ through the correspondence $\gamma_{f}:[0,1]\rightarrow i\big(\tilde{G}(V,k)\big)$,
$\gamma_{f}(t)=p\Big((e_{1}+tf(e_{1}))\wedge...\wedge(e_{k}+tf(e_{k}))\Big)$.
Clearly this identification is independent of the choice of $e_{1},...,e_{k}$.
Thus, there is a canonic identification $T_{\Lambda}\tilde{G}(V,k)\simeq Hom(\Lambda,V/\Lambda)$.

\subsubsection{Norms on spaces of operators}

Let $A,B$ be two linear spaces. When given an arbitrary norm $\beta$
on $Hom(A,B)$, one immediately obtains a norm $\overline{\beta}$
on $Hom(B^{*},A^{*})$ by letting $T\mapsto T^{*}$ be an isometry,
and the dual norm $\beta^{*}$ on $Hom(B,A)$, defined by trace duality.
It is immediate that $\overline{\beta^{*}}=\overline{\beta}^{*}$
on $Hom(A^{*},B^{*})$. Note that for a normed space $V$, the nuclear
(projective) norm $\beta=\|\bullet\|_{N}$ on $Hom(V,V)$ satisfies
that $\overline{\beta}$ is again the nuclear norm on $Hom(V^{*},V^{*})$.\\

Given a norm $\beta$ on $Hom(V,V)$ and a subspace $\Lambda\subset V$,
one has the natural inclusion map $Hom(V/\Lambda,\Lambda)\subset Hom(V,V)$,
and a quotient map $Hom(V,V)\twoheadrightarrow Hom(\Lambda,V/\Lambda)$.
Denote the induced subspace and quotient space norms by $\beta_{i}$
and $\beta_{\pi}$, respectively. It is immediate that $\overline{\beta_{\pi}}=\overline{\beta}_{\pi}$
on $Hom((V/\Lambda)^{*},\Lambda^{*})$, while $\overline{\beta_{i}}=\overline{\beta}_{i}$
on $Hom(\Lambda^{*},(V/\Lambda)^{*})$, $(\beta_{i})^{*}=(\beta^{*})_{\pi}$
on $Hom(\Lambda,V/\Lambda)$, and $(\beta_{\pi})^{*}=(\beta^{*})_{i}$
on $Hom(V/\Lambda,\Lambda)$.

\subsection{Definition of the Finsler structure}
\begin{defn}
For an arbitrary norm $\beta$ on $Hom(V,V)$, $(\tilde{G}(V,k),\phi_{\beta})$
is the Finsler manifold which has the quotient norm $\beta_{\pi}$
on the tangent spaces\\
$T_{\Lambda}\tilde{G}(V,k)=Hom(\Lambda,V/\Lambda)$. If $\beta$ is
smooth and strictly convex, we get a smooth Finsler manifold. Orientation
reversal on subspaces defines an antipodal map on $\tilde{G}(V,k)$
which is an isometry of $\phi_{\beta}$.\end{defn}
\begin{rem}
By trace duality, the cotangent spaces
\[
T_{\Lambda}^{*}\tilde{G}(V,k)=Hom(V/\Lambda,\Lambda)\subset Hom(V,V)
\]
 are equipped with the dual norm, which is $(\beta^{*})_{i}$. In
the following, we will often consider a cotangent vector $T\in T_{\Lambda}^{*}\tilde{G}(V,k)$
simply as an element of $Hom(V,V)$.
\end{rem}

\begin{rem}
As before, when we deal with the girth or Holmes-Thompson volume of
the $(\tilde{G}(V,k),\phi_{\beta})$, one can omit smoothness and
strict convexity assumptions on $\beta$.\end{rem}
\begin{example}
For a Euclidean space $V$ and $\beta=\|\|_{HS}$ the Hilbert-Schmidt
norm on $Hom(V,V)$, $(\tilde{G}(V,k);\phi_{HS})$ is the standard
$SO(n)$-invariant Riemannian structure.
\end{example}

\begin{rem}
One can consider also the following more functorial construction:
Let $K,L\subset V$ be two symmetric convex bodies. For a subspace
$\Lambda\subset V$, denote by $\Lambda_{K}$ the normed space $\Lambda$
with unit ball $\Lambda\cap K$. Consider some uniform crossnorm $\alpha$,
that is an assignment of a norm to $A^{*}\otimes B\simeq Hom(A,B)$
for pairs of isometry classes of finite dimensional normed spaces
$A,B$ (see \ref{sub:tensor norms} for a review of crossnorms). Important
examples are the operator (injective) norm $\|\|_{Op}$ and the nuclear
(projective) norm $\|\|_{N}$. Equipping the spaces $Hom(\Lambda_{K},V_{L}/\Lambda)$
with $\alpha$, $\tilde{(G}(V,k);\phi_{\alpha,K,L})$ becomes a Finsler
manifold. If $V$ is a given normed space, we take $K=L=\{\|x\|\leq1\}$
and denote $\phi_{\alpha,K,L}=\phi_{\alpha,V}$.\end{rem}
\begin{example}
For $k=1$, the Finsler structure on $\tilde{(G}(V,1);\phi_{\alpha,K,L})$
is independent of $\alpha$, and the induced Finsler structure on
$\tilde{\mathbb{P}}(V)$ will be denoted $\phi_{K,L}$. It also coincides
with $\tilde{(G}(V,1),\phi_{\beta})$ where $\beta$ is the nuclear
norm (or any other projective crossnorm) on $Hom(V_{K},V_{L})$.
\end{example}

For a smooth convex body $K\subset V$, $\partial K\simeq\tilde{\mathbb{P}}(V)$
as smooth manifolds. We will show that the Finsler manifold constructed
above generalizes the quotient Finsler structure of section \ref{sec:Projective-girth}.
\begin{prop}
As Finsler manifolds, one has $(\tilde{\mathbb{P}}(V),\phi_{K,L})=(\partial K,$$\phi_{L})$.\end{prop}
\begin{proof}
Fix $q\in\partial K$, denote $M=\langle q\rangle$, $\xi=\mathcal{L}(q)$,
$W=\{\xi=0\}$. A linear function $f:M\rightarrow V/M$ is uniquely
defined by $v=f(q)$. Let $w\in W$ be the unique vector with $v=Pr_{V/M}(w).$
Then $\phi_{K,L}(f)=\|f\|=\|v\|_{V_{L}/M}$. The curve $\gamma_{f}$
on the sphere is given by $\gamma_{f}(t)=\frac{q+tw}{\|q+tw\|}$ and
\[
\dot{\gamma}_{f}(0)=w-q\frac{d}{dt}\Bigg|_{t=0}\|q+tw\|=w
\]
so that $\phi_{L}(\dot{\gamma}_{f}(0))=\|Pr_{V/M}w\|_{V_{L}/M}=\|v\|_{V_{L}/M}$.
\end{proof}

\subsection{Main theorems}

Recall that $V$ is an oriented vector space.
\begin{prop}
\label{symmetry}For any norm $\beta$ on $Hom(V,V)$, the Finsler
manifolds $(\tilde{G}(V,n-k),\phi_{\beta})$ and $(\tilde{G}(V^{*},k),\phi_{\overline{\beta}})$
are canonically isometric. \end{prop}
\begin{proof}
The natural identification $A:\tilde{G}(V,n-k)\rightarrow\tilde{G}(V^{*},k)$
defined by $\Lambda\mapsto(V/\Lambda)^{*}$ (with the orientation
on $(V/\Lambda)^{*}$ induced by that of $\Lambda$) has the differential
$D_{\Lambda}(A):Hom(\Lambda,V/\Lambda)\rightarrow Hom((V/\Lambda)^{*},\Lambda^{*})$
given by $(D_{\Lambda}A)(f)=f^{*}$, which is by definition an isometry. \end{proof}
\begin{cor}
Assume $\alpha$ is a symmetric crossnorm, $K,L\subset V$ convex
symmetric bodies. Then $(\tilde{G}(V,n-k),\phi_{\alpha,K,L})$ and
$(\tilde{G}(V^{*},k),\phi_{\alpha,L^{o},K^{o}})$ are canonically
isometric. 
\end{cor}
Once we have made those observations, we will be concerned from now
on with various correspondences between $\tilde{G}(V,k)$ and $\tilde{G}(V,n-k)$.
We will assume without loss of generality that $2k\leq n$.
\begin{thm}
\label{thm:Holmes-Thompson}Fix any norm $\beta$ on $Hom(V,V)$.
Then

\[
v_{HT}(\tilde{G}(V,k),\phi_{\beta})=v_{HT}(\tilde{G}(V,n-k),\phi_{\beta})
\]
where $v_{HT}$ denotes the Holmes-Thompson volume.\end{thm}
\begin{proof}
First, assume that $\beta$ is smooth and strictly convex. 

For any $1\leq l<n$, let $s=\min(l,n-l)$ and consider the filtration
$T^{*}(\tilde{G}(V,l))=C_{s}^{l}\supset C_{s-1}^{l}\supset...\supset C_{1}^{l}\supset C_{0}^{l}$
where $C_{r}^{l}=\{(\Lambda,T)\in T^{*}\tilde{G}(V,l):rank(T:V\rightarrow V)\leq r\}$
are closed submanifolds. Also, define $E_{r}^{l}=C_{r}^{l}\backslash C_{r-1}^{l}$
- the points of rank $r$ in $T^{*}\tilde{G}(V,l)$, which is an open
submanifold in $C_{r}^{l}$ . It is easy to see that $C_{r}^{l}$
is connected for all $r$, and $E_{s}^{l}$ is open and dense.\\

We identify $T\in T_{\Lambda}^{*}\tilde{G}(V,l)=Hom(V/\Lambda,\Lambda)$
with $T:V\rightarrow V$ such that $Im(T)\subset\Lambda\subset Ker(T)$.
The group $GL(V)$ acts on $\tilde{G}(V,l)$. Therefore (see \ref{sub:A-Canonic-Structure on cotangent bund}),
we get an induced Hamiltonian action of $GL(V)$ on $T^{*}\tilde{G}(V,l)$,
which is given explicitly by $U(\Lambda,T)=(U\Lambda,UTU^{-1})$.
\\

We claim that the $GL(V)$-orbits are precisely $E_{r}^{l}$, $0\leq r\leq l$.
Indeed, all operators $T:V\rightarrow V$ satisfying $T^{2}=0$ and
$rank(T)=r$ are conjugate to each other, so it remains to show that
$(\Lambda,T)$ and $(\Lambda',T)$ lie in the same orbit. Choosing
some complement $Ker(T)\oplus W=V$, define $U:V\to V$ so that $U|_{Im(T)}=Id$,
$U(\Lambda)=\Lambda'$, $U(Ker(T))=Ker(T)$ and $U|_{W}=Id$. Then
obviously $U\in GL(V)$, and $UTU^{-1}=T$, proving the claim. The
corresponding orbits of the coadjoint action of $GL(V)$ on $\mathfrak{gl}(V)^{*}\backsimeq\mathfrak{gl}(V)$
are simply $A_{r}=\{T:T^{2}=0,rank(T)=r\}$, which are equipped with
Kirillov's symplectic form.\\

The momentum map $\mu:T^{*}\tilde{G}(V,l)\rightarrow\mathfrak{gl}(V)^{*}$
is given for $X\in\mathfrak{gl}(V)$ by
\[
\langle\mu(\Lambda,T),X\rangle=tr(T\underline{X}_{\Lambda})
\]
where $\underline{X}_{\Lambda}$ denotes the infinitesimal action
(fundamental vector field) of $X$ at $\Lambda$. Thus after identifying
$\mathfrak{gl}(V)^{*}$ with $\mathfrak{gl}(V)$ by trace duality,
$\mu(\Lambda,T)=T$. \\

It follows from \ref{cor:moment map is symplectomorphism} that $\mu_{l}:E_{s}^{l}\to A_{s}$
is locally a symplectomorphism, which is clearly surjective and 2-to-1
(for instance, if $l\leq n/2$ then $\mu_{l}(Im(T),T)=T$, and there
are two possible orientations for $Im(T)$). 

We would like to find a diffeomorphism $\Phi$ that makes the following
diagram commutative:

\[\xymatrix{ E_k^k\ar[dr]_{\mu_k}\ar[rr]^\Phi & &E_k^{n-k}\ar[dl]^{\mu_{n-k}}\\& A_k& }\]

It would immediately follow that $\Phi$ is a symplectomorphism. Define
\[
\Phi(Im(T),T)=(Ker(T),T)
\]
taking the orientation on $Ker(T)$ so that $T:V/Ker(T)\to Im(T)$
is orientation preserving. It is straightforward to verify that $\Phi$
satisfies all conditions, so $\Phi:E_{k}^{k}\to E_{k}^{n-k}$ is an
isomorphism of symplectic manifolds. Moreover, it obviously preserves
the norm $\phi_{\beta}^{*}$. This proves equality of volumes, since
for any Finsler metric $\phi$, 
\[
v_{HT}(\tilde{G}(V,l),\phi)=\frac{1}{(l(n-l))!}\int_{B^{*}(\tilde{G}(V,l),\phi)}\omega^{l(n-l)}=
\]
\[
=\frac{1}{(l(n-l))!}\int_{B^{*}(\tilde{G}(V,l),\phi)\cap E_{s}^{l}}\omega^{l(n-l)}
\]
Finally, the result for arbitrary norms $\beta$ follows by approximation,
and by continuity of the Holmes-Thompson volume w.r.t. $\beta$.\end{proof}
\begin{rem}
It follows by Remark \ref{omega_implies_alpha} that $\Phi^{*}\alpha_{2}=\alpha_{1}$.
\end{rem}

\begin{rem}
In the case $k=1$, it follows from the proof that the cotangent bundles
are symplectomorphic outside the zero section, and the associated
Hamiltonians are respected. It follows by Corollary \ref{symplectomorphism_implies_spectrum}
that the length spectra, as well as the symmetric length spectra,
coincide. Together with Proposition \ref{symmetry}, this generalizes
Theorems \ref{thm:The-co-sphere-bundles} and \ref{thm:(Dual-spheres-have}
to arbitrary norms $\beta$ on $Hom(V,V)$. 
\end{rem}

\begin{rem}
It is worth noting that $\Phi$ cannot be extended continuously outside
$E_{k}^{k}$ when $k<n/2$: for any $(\Lambda_{0},T_{0})\in T_{\Lambda_{0}}^{*}\tilde{G}(V,k)$
with $rank(T)<k$, one can always find two nearby points $(\Lambda_{0},T_{1}),(\Lambda_{0},T_{2})$
s.t. $rank(T_{1})=rank(T_{2})=k$ while $Ker(T_{1})$ and $Ker(T_{2})$
are far apart on $\tilde{G}(V,n-k)$. When $k=n/2$, $\Phi$ is just
the identity map.
\end{rem}
Before we proceed to study the girth of Grassmannians, let us briefly
recall some terminology. For a Finsler manifold $(M,\phi)$ and a
curve $\gamma_{t}\in M$, we call the curve $\Gamma_{t}=(\gamma_{t},\mathcal{L}(\dot{\gamma}_{t}))\in T^{*}M$
its \emph{lift} to $T^{*}M$. A curve $\Gamma_{t}$ of such form we
call a \emph{lift curve. }
\begin{lem}
For a smooth, strictly convex norm $\beta$ on $Hom(V,V)$, the geodesics
in $\tilde{(G}(V,k),\phi_{\beta})$ lift to curves of constant rank
in $T^{*}\tilde{G}(V,k)$.\end{lem}
\begin{proof}
We use the notation of the proof of Theorem \ref{thm:Holmes-Thompson}. 

The level sets of $\mu:T^{*}\tilde{G}(V,l)\to\mathfrak{gl}(V)$ are
$Z_{T}=\{(\Lambda,T):Im(T)\subset\Lambda\subset Ker(T)\}\subset T^{*}\tilde{G}(V,l).$
It follows (see \ref{sub:The-momentum-mapping}) that the skew-orthogonal
space $(T_{(\Lambda,T)}Z_{T})^{\perp}=T_{(\Lambda,T)}E_{r}^{l}$ where
$r=rank(T)$. 

The Hamiltonian $H=\frac{1}{2}\phi_{\beta}^{*2}$ is constant on $Z_{T}$,
and so $X_{H}\Big|_{E_{r}^{k}}\in TE_{r}^{k}$. Since $C_{r}^{k}$
and $C_{r-1}^{k}$ are closed, the flow defined by $X_{H}$ leaves
$E_{j}^{k}$ invariant, which concludes the proof.
\end{proof}
This motivates the following definition: 
\begin{defn}
Fix a smooth, strictly convex norm $\beta$ on $Hom(V,V)$. The \emph{rank}
of a geodesic in $(\tilde{G}(V,k),\phi_{\beta})$ is the constant
rank of its lift to $T^{*}\tilde{G}(V,k)$ \emph{.}\end{defn}
\begin{example}
The girth of $(\tilde{G}(V,k),\phi_{HS})$ is attained on a geodesic
of rank 1 (which can be visualized as a rotation of a two dimensional
plane, while fixing all orthogonal directions).
\end{example}
We will make use of the following general fact
\begin{lem}
\label{prop:Legendre projection}For a normed $X$, and a subspace
$Y\subset X$, $\mathcal{L}_{Y}(y)=Pr_{Y^{*}}(\mathcal{L}_{X}(y)$)
for all $y\in Y$. In particular, taking $X=(Hom(V,V),\beta)$, $Y=Hom(V/\Lambda,\Lambda)$
one has $\mathcal{L}_{Hom(V/\Lambda,\Lambda)}(T)=Pr_{Hom(\Lambda,V/\Lambda)}(\mathcal{L}_{Hom(V,V)}(T))$. \end{lem}
\begin{proof}
This follows by an immediate verification of the definitions.
\end{proof}
We will also need a lemma from linear algebra:
\begin{lem}
\label{lem:Linear Algebra}Let $V$ be an $n$-dimensional real vector
space, and $T\in GL(V)$. Suppose $\Lambda\subset V$ is a subspace
with $\dim\Lambda=k$ and $T(\Lambda)=\Lambda$. Then there is a subspace
$\Omega\subset V$ s.t. $\dim\Omega=n-k$, $T(\Omega)=\Omega$ and
$\det T|_{\Lambda}\det T|_{\Omega}=\det T$.\end{lem}
\begin{proof}
Simply observe that for $T^{*}\in GL(V^{*})$ and $\Lambda^{\perp}=(V/\Lambda)^{*}$,
$T^{*}(\Lambda^{\perp})=\Lambda^{\perp}$. It is well known that $T^{*}$
and $T$ are conjugate over $\R$, i.e. $T^{*}=UTU^{-1}$ for some
invertible $U:V\to V^{*}$. Thus $\Omega=U^{-1}\Lambda^{\perp}$ is
invariant for $T$, and $\det T=\det T|_{\Lambda}\det T|_{V/\Lambda}=\det T|_{\Lambda}\det T^{*}|_{\Lambda^{\perp}}$.\end{proof}
\begin{thm}
Fix a smooth, strictly convex norm $\beta$ on $Hom(V,V)$. Then there
exists a bijection between the closed geodesics of $\tilde{(G}(V,k),\phi_{\beta})$
and those of $\tilde{(G}(V,n-k),\phi_{\beta})$ which respects length
and rank. Moreover, symmetric geodesics correspond to symmetric geodesics.\end{thm}
\begin{proof}
We again use the notation of the proof of Theorem \ref{thm:Holmes-Thompson}.
Suppose that for some $1\leq r\leq k$, $\gamma_{t}=(\Lambda_{t},T_{t})\subset E_{r}^{k}$
, $0\leq t\leq L$ is a characteristic curve, i.e. the lift of a geodesic
$\Lambda_{t}$ of length $L$ and rank $r$ in $\tilde{(G}(V,k),\phi_{\beta})$
with arc-length parametrization, such that $\Lambda_{L}=\Lambda_{0}$
(closed geodesic - referred to as the first case) or $\Lambda_{L}=\overline{\Lambda_{0}}$
(half of a closed symmetric geodesic - referred to as the second case).
In the second case, the extension to a full symmetric geodesic is
given by $\Lambda_{t}=\overline{\Lambda_{t-L}}$ for $L\leq t\leq2L$.
The parameter $t$ is taken mod $L$ in the first case, and mod $2L$
in the second case. In both cases, $T_{t+L}=T_{t}$. Denote $I_{t}=Im(T_{t})$,
$K_{t}=Ker(T_{t})$, and define $S_{t}=\mathcal{L}_{Hom(V,V)}(T_{t})\in Hom(V,V)$.
Then $\beta(S_{t})=1$, and it follows from Lemma \ref{prop:Legendre projection}
that $\dot{\Lambda}_{t}=\Pr{}_{Hom(\Lambda_{t},V/\Lambda_{t})}(S_{t})$.
Define $B_{t}\in Hom(V,V)$ by the differential equation $\dot{B}_{t}=S_{t}B_{t}$,
with $B_{0}=Id$. Note that $B_{t}\in GL^{+}(V)$ for all $t$, and
$B_{t+L}=B_{t}B_{L}$.\\

First, observe that $B_{t}(\Lambda_{0})=\Lambda_{t}$. This is evident
by taking $e_{1}(0),...,e_{k}(0)$ a basis of $\Lambda_{0}$, and
$e_{l}(t)=B_{t}e_{l}(0)$. Then $\dot{e}_{l}(t)=S_{t}e_{l}(t)$, as
required.\\

We next claim that $B_{t}(I_{_{0}})=I_{t}$ and $B_{t}(K_{0})=K_{t}$.
This can be seen as follows: By Corollary \ref{cor:Hamiltonian flow moment map},
$\mu_{k}(\Gamma_{t})$ is a flow curve in $A_{r}$ for the Hamiltonian
$H(T)=\frac{1}{2}\beta^{*}(T)^{2}$. Since $A_{r}$ and $E_{r}^{r}$
are locally symplectomorphic through $\mu_{r}$, the curve $(I_{t},T_{t})\in E_{r}^{r}$
is characteristic (strictly speaking one first has to fix some orientation
on $I_{0}$, but we only consider local properties of the curve such
as it being characteristic, or a lift curve). In particular, it is
a lift curve, so $T_{t}=\mathcal{L}_{Hom(I_{t},V/I_{t})}(\dot{I}_{t})\Rightarrow\dot{I}_{t}=\Pr_{Hom(I_{t},V/I_{t})}(\mathcal{L}_{Hom(V,V)}T_{t})=\Pr_{Hom(I_{t},V/I_{t})}S_{t}$.
This readily implies as above that $I_{t}=B_{t}(I_{0})$. The proof
that $K_{t}=B_{t}(K_{0})$ is identical.\\

Denote by $\tilde{B}_{t}:K_{0}/I_{0}\to K_{t}/I_{t}$ the operator
induced from $B_{t}:V\to V$, and fix some orientation on $I_{0}.$
Then $K_{0}$ inherits an orientation, and so do $K_{t}=B_{t}(K_{0}),$
$I_{t}=B_{t}(I_{0})=V/K_{t}$ (equalities of oriented spaces). In
particular, $K_{L}/I_{L}$ and $K_{0}/I_{0}$ coincide as oriented
vector spaces (in fact, this orientation is independent of the one
on $I_{0}$), and $\tilde{B}_{L}\in GL^{+}(K_{0}/I_{0})$. Now consider
$\tilde{\Lambda}_{0}=\Lambda_{0}/I_{0}\subset K_{0}/I_{0}$ which
inherits an orientation from $(\Lambda_{0},I_{0})$, and $\dim\tilde{\Lambda}_{0}=k-r$.
\\

Apply Lemma \ref{lem:Linear Algebra} to conclude the existence of
a subspace $\tilde{\Omega}_{0}\subset K_{0}/I_{0}$ s.t. $\dim\tilde{\Omega}_{0}=n-k-r$
and $\tilde{B}_{L}(\tilde{\Omega}_{0})=\tilde{\Omega}_{0}$ as unoriented
spaces. Moreover, one can write 
\[
\sign\det\tilde{B}_{L}|_{\tilde{\Omega}_{0}}=\sign\det\tilde{B}_{L}|_{\tilde{\Lambda}_{0}}=\sign\det B_{L}|_{\Lambda_{0}}\sign\det B_{L}|_{I_{0}}
\]
Fix some orientation on $\tilde{\Omega}_{0}$, and define $\Omega_{0}=Pr_{K_{0}/I_{0}}^{-1}(\tilde{\Omega}_{0})$
with the induced orientation. Note that 
\[
B_{L}(\Omega_{0})=B_{L}(Pr_{K_{0}/I_{0}}^{-1}(\tilde{\Omega}_{0}))=Pr_{K_{0}/I_{0}}^{-1}(\tilde{B}_{L}\tilde{\Omega}_{0})=\Omega_{0}
\]
(ignoring orientations), while 
\[
\sign\det B_{L}|_{\Omega_{0}}=\sign\det B_{L}|_{\tilde{\Omega}_{0}}\sign\det B_{L}|_{I_{0}}=\sign\det B_{L}|_{\Lambda_{0}}
\]
i.e. $B_{L}(\Omega_{0})=\Omega_{0}$ (first case) or $B_{L}(\Omega_{0})=\overline{\Omega_{0}}$
(second case).\\

Let $\Gamma_{t}=(\Omega_{t},T'_{t})$ be the unique characteristic
curve through ($\Omega_{0},T_{0})$. According to Lemma \ref{cor:Hamiltonian flow moment map},
$\Gamma_{t}$ is mapped by $\mu_{k}$ to $T_{t}\in A_{r}$ so $T'_{t}=T_{t}$;
and $\Gamma_{t}$ is also a lift curve, so $\dot{\Omega}_{t}=Pr_{Hom(\Omega_{t},V/\Omega_{t})}(S_{t})$
and $\Omega_{t}=B_{t}(\Omega_{0})$ as before. We conclude that $\Omega_{t}$,
$0\leq t\leq L$ is a geodesic in $\tilde{(G}(V,n-k),\phi_{\beta})$
of length $L$ and rank $r$, which is closed (in the first case)
or constitutes half of the closed symmetric geodesic $\Omega_{t}=B_{t}\Omega_{0}$,
$0\leq t\leq2L$ (second case). Finally, it remains to note that we
may choose $\Omega_{0}=\Omega_{0}(\Lambda_{t})$ in a shift invariant
manner, i.e. in such a way that $\Omega_{0}(\Lambda_{t+T})=\Omega_{T}$
for all $T$. This concludes the proof.\end{proof}
\begin{rem}
The same proof shows also the existence of a correspondence between
geodesics joining antipodal points.\end{rem}
\begin{cor}
For any norm $\beta$ on $Hom(V,V)$, \textup{$(\tilde{G}(V,k),\phi_{\beta})$}\textup{\emph{
}}and $(\tilde{G}(V,n-k),\phi_{\beta})$ have equal girth.
\end{cor}

\begin{cor}
Assume $\alpha$ is a projective crossnorm, $K,L\subset V$ convex
symmetric bodies. Then $(\tilde{G}(V,k),\phi_{\alpha,K,L})$ and $\tilde{(G}(V,n-k),\phi_{\alpha,K,L})$
have equal Holmes-Thompson volume, and equal girth. If $\alpha$ on
$Hom(V_{K},V_{L})$ is smooth and strictly convex, the length spectra
and symmetric length spectra coincide together with rank.
\end{cor}

\begin{cor}
Assume $\alpha$ is a symmetric and projective crossnorm, $V$ a normed
space. Then $(\tilde{G}(V,k),\mbox{\ensuremath{\phi}}_{V,\alpha})$
and $(\tilde{G}(V^{*},k),\mbox{\ensuremath{\phi}}_{V^{*},\alpha})$
have equal Holmes-Thompson volume and equal girth.
\end{cor}
\appendix

\section{Appendix}

\subsection{\label{sub:tensor norms}A brief overview of tensor norms:}

For further details we refer to \cite{Ry}. A uniform crossnorm $\alpha$
is an assignment of a norm to $A^{*}\otimes B\simeq Hom(A,B)$ for
pairs of isometry classes of finite dimensional normed spaces $A,B$,
s.t. rank 1 operators have the natural norm: $\alpha(a^{*}b)=\|a^{*}\|\|b\|$.
Important examples are the operator (injective) norm $\|\|_{Op}$
and the nuclear (projective) norm $\|\|_{N}$. One can also introduce
the dual uniform crossnorm $\alpha^{*}$, given by trace duality:
$T\in(Hom(X,Y),\alpha)$ and $S\in(Hom(Y,X),\alpha^{*})$ are paired
by $(T,S)\mapsto tr(ST).$ For $\alpha=\|\|_{Op}$, one has $\alpha^{*}=\|\|_{N}$.
\begin{enumerate}
\item A uniform crossnorm $\alpha$ is symmetric if the adjoint map
\[
^{*}:(Hom(A,B),\alpha)\rightarrow(Hom(B^{*},A^{*}),\alpha)
\]
is an isometry for all pairs $A,B$. If $\alpha$ is symmetric, so
is $\alpha^{*}$. 
\item $\alpha$ is injective if for all quadruples $(A_{1}\subset A,B_{1}\subset B)$,
the natural injection $Hom(A/A_{1},B_{1})\hookrightarrow Hom(A,B)$
is an injection of $\alpha$-normed spaces. 
\item $\alpha$ is projective if for all quadruples $(A_{1}\subset A,B_{1}\subset B)$,
the natural projection $Hom(A,B)\twoheadrightarrow Hom(A_{1},B/B_{1})$
is a projection of $\alpha$-normed spaces.
\item $\alpha$ is projective if and only if $\alpha^{*}$ is injective. 
\item Thus $\|\|_{Op}$ is symmetric and injective, while $\|\|_{N}$ is
symmetric and projective. 
\end{enumerate}

\subsection{\label{sub:The-momentum-mapping}The momentum map}

\subsubsection{Generalities on momentum map}

Most of the following can be found in any textbook on Hamiltonian
dynamics, see for instance \cite{Au} . It appears here to make the
exposition self contained, to fix notation, and also to prove several
lemmas which we couldn't find in the precise form which we need to
apply.\\
\\
We are given a Lie group $G$, its Lie algebra $\mathfrak{g}$, and
a symplectic manifold $W$ with a symplectic action of $G$.
\begin{claim}
Given maps $\mu:W\rightarrow\mathfrak{g}^{*}$ and $\tilde{\mu}:\mathfrak{g}\rightarrow C^{\infty}(W)$
satisfying $\tilde{\mu}(X)(p)=\langle\mu(p),X\rangle$, then $\omega(\underline{X}_{p},\bullet)=\langle D_{p}\mu(\bullet),X\rangle$
for all $X\in\mathfrak{g}$ if and only if $\tilde{\mu}(X)$ is a
Hamiltonian map for $\underline{X}$, i.e. $d_{p}(\tilde{\mu}(X))=\omega(\underline{X}_{p},\bullet)$.\end{claim}
\begin{proof}
Calculation: $\mu(p)(X)=\tilde{\mu}(X)(p)$, after differentiating
by $p$ one has $\langle D_{p}\mu(\bullet),X\rangle=d_{p}(\tilde{\mu}(X))$.\end{proof}
\begin{defn}
Under these conditions, the action of $G$ is \emph{Hamiltonian},
and $\mu$ is called a \emph{momentum map}. In particular, for $H=\langle X,\mu(\bullet)\rangle=\tilde{\mu}\circ X:W\to\mathbb{R}$,
$X_{H}=\underline{X}$. 
\end{defn}
For a general (not necessarily equivariant) momentum map, we can describe
the image of the differential $D_{p}\mu$: Since $\langle D_{p}\mu(v),X\rangle=\langle v,i_{\underline{X}_{p}}\omega\rangle$,
$D_{p}\mu:T_{p}W\rightarrow\mathfrak{g}^{*}$ and $i_{\underline{X}_{p}}:\mathfrak{g}\rightarrow T_{p}^{*}W$
are dual maps, so 
\[
Im(D_{p}\mu)=Ann(Ker(i_{\underline{X}_{p}}))=Ann(\{X:\underline{X}_{p}=0\})
\]
and if $\mu(p)=\xi$ 
\[
T_{p}(\mu^{-1}(\xi))=Ker(D_{p}\mu)=Ann(Im(i_{\underline{X}_{p}}))=\{v:\omega(v,\underline{X}_{p})=0\}=T_{p}(Gp)^{\perp}
\]
In particular, $rank(D_{p}\mu)=codim(\{X:\underline{X}_{p}=0\})=dim(Gp)$.

\begin{defn}
If the fundamental vector fields on $W$ and $\mathfrak{g}^{*}$ for
the coadjoint action satisfy $\underline{X}_{\mu(p)}=D_{p}\mu(\underline{X}_{p})$
for all $p\in W$, the momentum map is called \emph{equivariant}.

\end{defn}
\begin{claim}
The momentum map is equivariant if and only if $\omega(\underline{X}_{p},\underline{Y}_{p})=\langle\mu(p),[X,Y]\rangle$
for all $X,Y\in\mathfrak{g}$, $p\in W$.\end{claim}
\begin{proof}
Take any $Z\in\mathfrak{g}$. Then the co-adjoint infinitesimal action
is given by 
\[
\langle\underline{X}_{\mu(p)},Z\rangle=\langle\mu(p),[Z,X]\rangle
\]
while by definition of momentum map
\[
\langle D_{p}\mu(\underline{X}_{p}),Z\rangle=\omega(\underline{Z}_{p},\underline{X}_{p})
\]
Thus equivariance of $\mu$ amounts to equality of the two right hand
sides.\end{proof}
\begin{cor}
\label{cor:equivariant-momentum}Suppose the momentum map $\mu$ commutes
with the action of $G$: $\mu(gp)=Ad_{g}^{*}(\mu(p))$. Then $\mu$
is an equivariant momentum map. If $G$ is connected, the reverse
implication also holds.\end{cor}
\begin{proof}
Take $g(t)=exp(tX)$ and differentiate: $D_{p}\mu(\underline{X}_{p})=\underline{X}_{\mu(p)}$,
and the Claim above applies. Connectedness allows to integrate this
equation.\end{proof}
\begin{rem}
When $\mu(gp)=Ad_{g}^{*}(\mu(p))$, we refer to $\mu$ as a $G$-equivariant
momentum map. By the corollary, $G$-equivariance implies equivariance,
and the two notions coincide for connected groups $G$.\end{rem}
\begin{fact}
The co-adjoint orbit $G\xi$ is naturally a symplectic manifold, with
Kirillov's symplectic form given by $\omega_{\xi}(\underline{X}_{\xi},\underline{Y}_{\xi})=\langle\xi,[X,Y]\rangle$
for $X,Y\in\mathfrak{g}$. The action of $G$ on $G\xi$ is Hamiltonian,
with $G$-equivariant momentum map given by the inclusion $G\xi\subset G$.\end{fact}
\begin{cor}
\label{cor:moment map is symplectomorphism}Let $W$ be a symplectic
manifold equipped with a Hamiltonian action of $G$ and an equivariant
momentum map $\mu$. Suppose $\mu(p)=\xi$ where $p\in W$, $\xi\in\mathfrak{g^{*}}.$
Then , $\omega_{p}(\underline{X}_{p},\underline{Y}_{p})=\omega_{\xi}(\underline{X}_{\xi},\underline{Y}_{\xi})$
for all $X,Y\in\mathfrak{g}$. If moreover the action of $G$ is transitive
at $p$ (i.e. $T_{p}(Gp)=T_{p}W$), then $\mu^{^{*}}\omega_{\xi}=\omega_{p}$\end{cor}
\begin{proof}
By equivariance of $\mu$, we get $\omega_{p}(\underline{X}_{p},\underline{Y}_{p})=\langle\xi,[X,Y]\rangle=\omega_{\xi}(\underline{X}_{\xi},\underline{Y}_{\xi})$.
The last part amounts to the verification $\omega_{p}(\underline{X}_{p},\underline{Y}_{p})=\omega_{\xi}(D_{p}\mu\underline{X}_{p},D_{p}\mu\underline{Y}_{p})$
for all $X,Y\in\mathfrak{g}$, which follows from the first part again
by equivariance of $\mu$.\end{proof}
\begin{cor}
\label{cor:Hamiltonian flow moment map}Let $G$ be a Lie group, $H:\mathfrak{g}^{*}\to\R$
any smooth function, and fix $\xi\in\mathfrak{g}^{*}$. Let $W$ be
a symplectic manifold equipped with a Hamiltonian action of $G$ and
momentum map $\mu$. Suppose $\mu(p)=\xi$. Consider the Hamiltonian
$\mu^{*}H=H\circ\mu$ on $W$. Denote the $\mu^{*}H$-flow on $W$
by $\phi_{t}$, and the $H$-flow on $G\xi$ by $\psi_{t}$ . Then 

(1) $\phi_{t}(p)\in Gp$. 

(2) If $\mu$ is equivariant, then $\mu\phi_{t}(p)=\psi_{t}(\xi)$.\end{cor}
\begin{proof}
(1) Simply note that $\mu^{*}H$ is constant along level sets of $\mu$,
so $X_{\mu^{*}H}\in T_{p}(\mu^{-1}(\xi))^{\perp}=T_{p}(Gp)$. For
(2), we should verify that 
\[
D_{p}\mu(X_{\mu^{*}H})=X_{H}
\]
Take any $Y\in\mathfrak{g}$ and verify that $\langle X_{H},Y\rangle=\langle D_{p}\mu(X_{\mu^{*}H}),Y\rangle$.
We know that 
\[
\langle D_{p}\mu(X_{\mu^{*}H}),Y\rangle=-\omega(X_{\mu^{*}H},\underline{Y}_{p})=-d_{p}(H\circ\mu)(\underline{Y}_{p})=
\]
\[
=-dH(D_{p}\mu\underline{Y}_{p})=-\omega_{\xi}(X_{H},D_{p}\mu\underline{Y}_{p})
\]
Now by equivariance, $D_{p}\mu\underline{Y}_{p}=\underline{Y}_{\xi}$,
while $X_{H}=\underline{X}_{\xi}$ for some $X\in\mathfrak{g}$ (since
it is tangent to the co-adjoint orbit of $\xi$), so we need only
check that $-\omega_{\xi}(\underline{X}_{\xi},\underline{Y}_{\xi})=\langle\underline{X}_{\xi},Y\rangle$,
and both sides equal $\langle\xi,[Y,X]\rangle$.%

\end{proof}

\subsubsection{\label{sub:A-Canonic-Structure on cotangent bund}A canonic structure
on the cotangent bundle }

Given a Lie group $G$ acting on a smooth manifold $M$, one can naturally
extend the action of $G$ to $T^{*}M$: $\hat{g}(q,p)=(gq,v\mapsto p(dg^{-1}(v)))$.
Then one easily verifies that this action preserves the canonic 1-form
$\alpha$: since for $\xi\in T_{(q,p)}(T^{*}M)$ one has $dqd\hat{g}(\xi)=dgdq(\xi)$,
\[
\alpha_{\hat{g}(q,p)}(d\hat{g}(\xi))=p'(dqd\hat{g}(\xi))=p(dg^{-1}dqd\hat{g}(\xi))=p(dq(\xi))=\alpha(\xi)
\]
Then given $X\in\mathfrak{g}$, $\tilde{\mu}(X)(q,p)=p(\underline{X}_{q})$
is a Hamiltonian function for $\underline{X}_{q}$: 
\[
\omega(\bullet,\hat{\underline{X}_{q}})=i_{\hat{\underline{X}_{q}}}(d\alpha)=L_{\hat{\underline{X}_{q}}}\alpha+d(i_{\hat{\underline{X}_{q}}}\alpha)
\]
The first summand is 0 since the action of $G$ preserves $\alpha$,
and $i_{\hat{\underline{X}_{q}}}\alpha=\alpha(\hat{\underline{X}_{q}})=p(\underline{X}_{q})$;
so $\omega(\bullet,\hat{\underline{X}_{q}})=d(\tilde{\mu}(X))$ as
required. The momentum map itself can therefore be defined by
\[
\langle\mu(q,p),X\rangle=p(\underline{X}_{q})
\]
This momentum map is $G$-equivariant: 
\[
\langle\mu(g(q,p)),X\rangle=g_{*}p(\underline{X}_{gq})=p(dg^{-1}(\underline{X}_{gq}))
\]
while 
\[
\langle Ad_{g}^{*}\mu(q,p),X\rangle=\langle\mu(q,p),Ad_{g^{-1}}X\rangle=p(\underline{Ad_{g^{-1}}X}_{q})
\]
and $\underline{Ad_{g^{-1}}X}_{q}=\frac{d}{dt}(g^{-1}exp(tX)gq)=dg^{-1}(\underline{X}_{gq})$.
So, $\mu(g(q,p))=Ad_{g}^{*}\mu(q,p)$ and in particular by Corollary
\ref{cor:equivariant-momentum}, $\mu$ is equivariant.

\subsection{\label{sub:The-projective-girth of the square}The quotient girth
of the square}

Fix a coordinate system in $\R^{2}$, let $Q=[-1,1]\times[-1,1]$.
Let $\gamma$ be the boundary curve parametrized by angle $0\leq\alpha\leq2\pi$,
so $\gamma=(1,\tan\alpha)$ and $\dot{\gamma}=(0,\frac{1}{\cos^{2}\alpha})$
for $0\leq\alpha\leq\pi/4$. One immediately calculates that 

\[
g_{q}(Q)=8\int_{0}^{\pi/4}\frac{\det(\gamma(\alpha),\dot{\gamma}(\alpha))}{\det(\gamma(\alpha),\gamma(3\pi/4))}d\alpha=8\int_{0}^{\pi/4}\frac{d\alpha}{\cos\alpha(\sin\alpha+\cos\alpha)}=8\log2=5.54..
\]

\lyxaddress{School of Mathematical Sciences, Tel-Aviv University, Tel Aviv 69978,
Israel. }

\lyxaddress{Email:faifmand@post.tau.ac.il}
\end{document}